\documentclass[twoside,12pt]{article}

\usepackage{times,amsmath,amssymb,amsthm,cite}

\newcommand{\subj}[1]{\par\noindent{\bf AMS Subject Classifications: }#1.}
\newcommand{\keyw}[1]{\par\noindent{\bf Keywords: }#1.}
\numberwithin{equation}{section}
\numberwithin{figure}{section}

\newtheorem{theorem}{Theorem}[section]
\newtheorem{lemma}[theorem]{Lemma}
\theoremstyle{definition}
\newtheorem{definition}[theorem]{Definition}
\theoremstyle{remark}
\newtheorem{remark}[theorem]{Remark}

\everymath{\displaystyle}
\date{}
\usepackage{mathrsfs,amsmath,mathtools}
\usepackage{url}

\newcommand{\adsa}
{\vspace{-1in}\normalsize\flushleft
This is a preprint of a paper whose final and definite form will be published in\\
Advances in Dynamical Systems and Applications, ISSN 0973-5321 ({\tt http://campus.mst.edu/adsa}).\\
\vspace{1mm}\hrule\vspace{5mm}
\renewcommand\thefootnote{{}}
\footnotetext{\noindent\tt Submitted 15/Mar/2014; Revised 23/Apr/2014; Accepted 12/May/2014.\par
\hspace*{8pt}Communicated by Sandra Pinelas}}

\begin{document}

\title{\adsa \center\Large\bf Noether's Theorem with Momentum
and Energy Terms for Cresson's Quantum Variational Problems}

\author{{\bf Gast\~{a}o S. F. Frederico}\\
Department of Mathematics, Federal University of St. Catarina\\
P.O. Box 476, 88040--900 Florian\'{o}polis, Brazil\\[0.2cm]
Department of Science and Technology\\
University of Cape Verde, Palmarejo, 279 Praia, Cape Verde\\
{\tt gastao.frederico@docente.unicv.edu.cv}
\and
{\bf Delfim F. M. Torres}\\
Center for Research and Development in Mathematics and Applications\\
Department of Mathematics, University of Aveiro, Aveiro, Portugal\\
{\tt delfim@ua.pt}}


\date{}

\maketitle


\thispagestyle{empty}


\begin{abstract}
We prove a DuBois--Reymond necessary optimality condition and
a Noether symmetry theorem to the recent quantum variational calculus of Cresson.
The results are valid for problems of the calculus of variations with
functionals defined on sets of nondifferentiable functions. As
an application, we obtain a constant of motion for a
linear Schr\"{o}dinger equation.
\end{abstract}


\subj{49K05, 49S05, 81Q05}

\keyw{Cresson's quantum calculus of variations, symmetries, constants of motion,
DuBois--Reymond optimality condition, Noether's theorem, Schr\"{o}dinger equation}


\section{Introduction}

Quantum calculus, sometimes called ``calculus without limits'',
is analogous to traditional infinitesimal calculus without
the notion of limits \cite{MR1865777}. Several dialects
of quantum calculus are available in the literature, including
Jackson's quantum calculus \cite{MR1865777,MyID:220},
Hahn's quantum calculus \cite{MyID:194,MR3028184,MyID:187},
the time-scale $q$-calculus \cite{MR1843232,MyID:146},
the power quantum calculus \cite{MyID:185},
and the symmetric quantum calculi \cite{MyID:222,MyID:229,MyID:246}.
Here we consider the recent quantum calculus of Cresson.

Motivated by Nottale's theory of scale relativity without the hypothesis
of space–time differentiability \cite{Nottale:1992,Nottale:1999}, Cresson
introduced in 2005 his quantum calculus on a set of H\"{o}lder functions \cite{MR2138974}.
This calculus attracted attention due to its
applications in physics and the calculus of variations,
and has been further developed by several different authors --
see \cite{MyID:131,MR2738030,MyID:111} and references therein.
Cresson's calculus of 2005 \cite{MR2138974} presents, however, some problems,
and in 2011 Cresson and Greff improved it \cite{MR2798411,Cresson:2011}.
It is this new version of 2011 that we consider here, a brief review of it
being given in Section~\ref{sec:2}. Along the text, by \emph{Cresson's calculus}
we mean this quantum version of 2011 \cite{MR2798411,Cresson:2011}.

There is a close connection between quantum calculus and the calculus of variations.
For the state of the art on the quantum calculus of variations
we refer the reader to the recent book \cite{book:quant}.
With respect to Cresson's approach, the quantum calculus of variations is still in its infancy:
see \cite{MyID:162,MyID:188,MyID:253,MR2798411,Cresson:2011,Dubois:2012}.
In \cite{MR2798411} a Noether type theorem is proved but only with the momentum term.
In \cite{Cresson:2011} nondifferentiable Euler--Lagrange equations are used in the study of PDEs.
It is proved that nondifferentiable characteristics for the Navier--Stokes equation correspond
to extremals of an explicit nondifferentiable Lagrangian system, and that the solutions of the
Schr\"{o}dinger equation are nondifferentiable extremals of the Newton Lagrangian.
Euler--Lagrange equations for variational functionals with Lagrangians containing
multiple quantum derivatives, depending on a parameter, or containing higher-order quantum
derivatives, are studied in \cite{MyID:162}. Variational problems with constraints,
with one and more than one independent variable, of first and higher-order type,
are investigated in \cite{MyID:188}. Recently, Hamilton--Jacobi equations were obtained
\cite{Dubois:2012} and problems of the calculus of variations and optimal control with time delay
were considered \cite{MyID:253}. Here we extend the available nondifferentiable Noether's theorem
of \cite{MR2798411} by considering invariance transformations that also change the time variable,
and thus obtaining not only the generalized momentum term of \cite{MR2798411}
but also a new energy term. For that we first obtain a new necessary optimality
condition of DuBois--Reymond type.

The text is organized as follows.
In Section~\ref{sec:2} we recall the notions and results
of Cresson's quantum calculus needed in the sequel.
Our main results are given in Section~\ref{sec:mr}:
the nondifferentiable DuBois--Reymond necessary optimality condition
(Theorem~\ref{theo:cdrnd}) and the nondifferentiable Noether type
symmetry theorem (Theorem~\ref{theo:tnnd}). We end with
an application of our results to the linear Schr\"{o}dinger equation
(Section~\ref{sec:appl}).


\section{Cresson's Quantum Calculus}
\label{sec:2}

We briefly review the necessary concepts and results of
the quantum calculus \cite{Cresson:2011}.
Let $\mathbb{X}^d$ denote the set $\mathbb{R}^{d}$ or $\mathbb{C}^{d}$,
$d \in \mathbb{N}$, and $I$ be an open set in $\mathbb{R}$
with $[t_1,t_2]\subset I$, $t_1<t_2$.
Denote by $\mathcal{G}\left(I,\mathbb{X}^d\right)$
the set of functions $f:I \rightarrow \mathbb{X}^d$
and by $\mathcal{C}^{0}\left(I,\mathbb{X}^d\right)$
the subset of functions of $\mathcal{G}\left(I,\mathbb{X}^d\right)$
that are continuous.

\begin{definition}[The $\epsilon$-left and $\epsilon$-right quantum derivatives \cite{Cresson:2011}]
Let $f\in \mathcal{C}^{0}\left(I, \mathbb{R}^{d}\right)$. For all
$\epsilon>0$, the $\epsilon$-left and $\epsilon$-right quantum derivatives of
$f$, denoted respectively by $\Delta_{\epsilon}^{-}f$
and $\Delta_{\epsilon}^{+}f$, are defined by
\begin{equation*}
\Delta_{\epsilon}^{-}f(t)=\frac{f(t)-f(t-\epsilon)}{\epsilon}
\quad \text{ and } \quad
\Delta_{\epsilon}^{+}f(t)=\frac{f(t+\epsilon)-f(t)}{\epsilon} \, .
\end{equation*}
\end{definition}

\begin{remark}
The $\epsilon$-left and $\epsilon$-right quantum derivatives
of a continuous function $f$ correspond to the classical derivative
of the $\epsilon$-mean function $f_{\epsilon}^{\sigma}$ defined by
\begin{equation*}
f_{\epsilon}^{\sigma}(t)=\frac{\sigma}{\epsilon}
\int_{t}^{t+\sigma\epsilon}f(s)ds\, ,
\quad \sigma=\pm \, .
\end{equation*}
\end{remark}

The next operator generalizes the classical derivative.

\begin{definition}[The $\epsilon$-scale derivative \cite{Cresson:2011}]
\label{def:qd}
Let $f\in \mathcal{C}^{0}\left(I,\mathbb{R}^{d}\right)$.
For all $\epsilon>0$, the $\epsilon$-scale
derivative of $f$, denoted by $\frac{\square_{\epsilon}f}{\square t}$,
is defined by
\begin{gather*}
\frac{\square_{\epsilon}f}{\square t}
=\frac{1}{2}\left[\left(\Delta_{\epsilon}^{+}f
+\Delta_{\epsilon}^{-}f\right)
+i\mu\left(\Delta_{\epsilon}^{+}f
-\Delta_{\epsilon}^{-}f\right)\right],
\end{gather*}
where $i$ is the imaginary unit and $\mu=\{-1,1,0,-i,i\}$.
\end{definition}

\begin{remark}
If $f$ is differentiable, then one can take the limit of the scale
derivative when $\epsilon$ goes to zero. We then obtain the
classical derivative $\frac{df}{dt}$ of $f$.
\end{remark}

We also need to extend the scale derivative to complex valued
functions.

\begin{definition}[See \cite{Cresson:2011}]
Let  $f\in \mathcal{C}^{0}\left(I,\mathbb{C}^{d}\right)$
be a continuous complex valued function.
For all $\epsilon>0$, the $\epsilon$-scale derivative of $f$,
denoted by $\frac{\square_{\epsilon}f}{\square t}$, is defined by
\begin{gather*}
\frac{{\square}_{\epsilon}f}{{\square}t}
=\frac{{\square}_{\epsilon}\textrm{Re}(f)}{\square t}
+i\frac{\square_{\epsilon}\textrm{Im}(f)}{\square t} \, ,
\end{gather*}
where $\textrm{Re}(f)$ and $\textrm{Im}(f)$ denote the real and
imaginary part of $f$, respectively.
\end{definition}

In Definition~\ref{def:qd}, the $\epsilon$-scale derivative
depends on $\epsilon$, which is a free parameter related to the
smoothing order of the function. This brings many difficulties in
applications to physics, when one is interested in particular
equations that do not depend on an extra parameter. To solve these
problems, the authors of \cite{Cresson:2011} introduced a procedure
to extract information independent of $\epsilon$ but related
with the mean behavior of the function.

\begin{definition}[See \cite{Cresson:2011}]
Let ${C^0_{conv}}\left(I\times
(0,1),\mathbb{R}^{d}\right)\subseteq {C^0}\left(I\times
(0,1),\mathbb{R}^{d}\right)$ be such that for any function $f \in
{C^0_{conv}}\left(I\times (0,1),\mathbb{R}^{d}\right)$ the
$\lim_{\epsilon\to 0}f(t,\epsilon)$
exists for any $t\in I$; and $E$ be a complementary of
${C^0_{conv}}\left(I\times (0,1),\mathbb{R}^{d}\right)$ in
${C^0}\left(I\times (0,1),\mathbb{R}^{d}\right)$. We define
the projection map $\pi$ by
$$
\begin{array}{lcll}
\pi: & {C^0_{conv}}\left(I\times (0,1),\mathbb{R}^{d}\right)
\oplus E & \to & {C^0_{conv}}(I\times \left(0,1),\mathbb{R}^{d}\right)\\
& f_{conv}+f_E  & \mapsto & f_{conv}
\end{array}
$$
and the operator $\left< \cdot \right>$ by
$$
\begin{array}{lcll}
\left< \cdot \right>: & {C^0}\left(I\times (0,1),\mathbb{R}^{d}\right)
& \to & {C^0}\left(I,\mathbb{R}^{d}\right)\\
& f & \mapsto & \left< f \right>: t\mapsto
\displaystyle\lim_{\epsilon\to 0}\pi(f)(t,\epsilon)\,.
\end{array}
$$
\end{definition}

The quantum derivative of $f$ without the dependence
of $\epsilon$ is introduced in \cite{Cresson:2011}.

\begin{definition}[See \cite{Cresson:2011}]
\label{def:ourHD}
The quantum derivative of $f$ in the space
$\mathcal{C}^{0}\left(I, \mathbb{R}^{d}\right)$
is given by the rule
\begin{equation}
\label{eq:scaleDer}
\frac{\Box f}{\Box t}=\left<
\frac{{\Box_{\epsilon}}f}{\Box t} \right>.
\end{equation}
\end{definition}

The scale derivative \eqref{eq:scaleDer} has some nice properties.
Namely, it satisfies a Leibniz and a Barrow rule. First let us
recall the definition of an $\alpha$-H\"{o}lderian function.

\begin{definition}[H\"{o}lderian function of exponent $\alpha$ \cite{Cresson:2011}]
Let $f\in C^0\left(I, \mathbb{R}^{d}\right)$. We say that $f$ is
$\alpha$-H\"{o}lderian, $0<\alpha<1$, if for all $\epsilon>0$ and
all $t$, $t'\in I$ there exists a constant $c>0$ such that
$|t-t'|\leqslant\epsilon$ implies
$\|f(t)-f(t')\|\leqslant c\epsilon^{\alpha}$,
where $\|\cdot\|$ is a norm in $\mathbb{R}^{d}$.
The set of  H\"{o}lderian functions of H\"{o}lder exponent $\alpha$,
for some $\alpha$, is denoted by $H^\alpha(I,\mathbb{R}^{d})$.
\end{definition}

In what follows, we frequently use $\square$ to denote the
scale derivative operator $\frac{{\square}}{{\square}t}$.

\begin{theorem}[The quantum Leibniz rule \cite{Cresson:2011}]
\label{theo:mult}
Let $\alpha+\beta>1$. For $f\in H^\alpha\left(I,
\mathbb{R}^{d}\right)$ and $g\in H^\beta\left(I,
\mathbb{R}^{d}\right)$, one has
\begin{equation}
\label{eq:mult}
\square(f\cdot g)(t)=\square f(t) \cdot
g(t)+f(t)\cdot\square g(t)\,.
\end{equation}
\end{theorem}

\begin{remark}
For $f\in \mathcal{C}^1\left(I, \mathbb{R}^{d}\right)$ and $g\in
\mathcal{C}^1\left(I, \mathbb{R}^{d}\right)$, one obtains from
\eqref{eq:mult} the classical Leibniz rule: $(f\cdot g)'=f'\cdot
g+f\cdot g'$. For convenience of notation, we sometimes write
\eqref{eq:mult} as $(f\cdot g)^\square(t)= f^\square(t)
\cdot g(t)+f(t)\cdot g^\square(t)$.
\end{remark}

\begin{theorem}[The quantum Barrow rule \cite{Cresson:2011}]
\label{Barrow} Let $f\in \mathcal{C}^0([t_1,t_2],\mathbb{R})$ be
such that $\Box f / \Box t$ is continuous and
\begin{equation}
\label{nec_condition}
\lim_{\epsilon\to0} \int_{t_{1}}^{t_{2}}
\left(\frac{\Box_\epsilon f}{\Box t}\right)_E(t)dt=0.
\end{equation}
Then,
\begin{equation}
\label{Barrow1}
\int^{t_{2}}_{t_{1}} \frac{\Box f}{\Box t}(t)\, dt=f(b)-f(a)\,.
\end{equation}
\end{theorem}

The next theorem gives the analogous of the derivative
of a composite function for the quantum derivative.

\begin{theorem}[See \cite{Cresson:2011}]
\label{theo:derivada}
Let $f \in C^{2}\left({\mathbb{R}^d\times I,\mathbb{R}}\right)$
and $x\in H^{\alpha}(\mathbb{R}^d,I)$ with $\frac{1}{2}\leq\alpha< 1$. Then,
\begin{multline*}
\frac{{\square}f}{{\square}t}(x(t),t)=\frac{\partial
f}{\partial t}(x(t),t)+\nabla_xf(x(t),t)\cdot\nabla_{\square}x(t) \\
 +\sum^{d}_{k=1}\sum^{d}_{j=1}\frac{1}{2}\frac{\partial
^{2}f}{\partial
x_kx_{j}}\left(x(t),t\right)a_{k,j}(x(t)),
\end{multline*}
where
\begin{equation*}
\nabla_{\square}x(t)=\left(\frac{\square x_1}{\square t}(t),
\ldots,\frac{\square x_n}{\square t}(t)\right)^T
\end{equation*}
and $a_{k,j}(x(t))$ denotes
$$
\left< \pi\left( \frac{\epsilon}{2}\left(
\left(\Delta^{+}_{\epsilon}x_k(t)\right)\left(\Delta^{+}_{\epsilon}x_k(t)\right)
(1+i\mu)-\left(\Delta^{-}_{\epsilon}x_k(t)\right)\left(\Delta^{-}_{\epsilon}x_k(t)\right)
(1-i\mu)\right)\right)\right>.
$$
\end{theorem}


\section{Main Results}
\label{sec:mr}

The classical Noether's theorem is valid along extremals $q$ which are $C^2$-differentiable.
The biggest class where a Noether type theorem has been proved for the classical
problem of the calculus of variations is the class of Lipschitz functions
\cite{CD:JMS:Torres:2004}. In this work we prove a more general
Noether type theorem, valid for nondifferentiable scale extremals.

In \cite{Cresson:2011} the calculus of variations with scale
derivatives is introduced and respective Euler--Lagrange equations
derived without the dependence of $\epsilon$.
In this section we obtain a formulation of Noether's
theorem for the scale calculus of variations. The proof of our
Noether's theorem is done in two steps: first we extend the
DuBois--Reymond condition to problems with scale derivatives
(Theorem~\ref{theo:cdrnd}); then, using this result, we obtain the
scale/quantum Noether's theorem (Theorem~\ref{theo:tnnd}).
The problem of the calculus of variations with scale derivatives is defined as
\begin{gather}
\label{Pe}
I[q(\cdot)] = \int_{a}^{b}
L\left(t,q(t),\square q(t)\right) dt \longrightarrow \min
\end{gather}
under given boundary conditions $q(a)=q_{a}$ and
$q(b)=q_{b}$, $(q(\cdot),\square q(\cdot)) \in
H^{2\alpha}$, $0<\alpha<1$. The Lagrangian $L
$ is assumed to be a $C^{1}$-function with
respect to all its arguments.

\begin{remark}
In the case of admissible differentiable functions $q(\cdot)$,
functional $I[q(\cdot)]$ in \eqref{Pe} reduces to the classical
variational functional of the fundamental problem
of the calculus of variations:
\begin{equation*}
I[q(\cdot)] = \int_a^b L\left(t,q(t),\dot{q}(t)\right) dt.
\end{equation*}
\end{remark}

\begin{theorem}[Nondifferentiable Euler--Lagrange equations \cite{Cresson:2011}]
\label{Thm:NonDtELeq}
Let $0<\alpha,\, \beta<1$ with $\alpha+\beta>1$.
If $q\in H^{\alpha}\left(I,\mathbb{R}^{d}\right)$
satisfies $\square q\in H^{\alpha}\left(I,
\mathbb{R}^{d}\right)$ and
$L\left(t,q(t),\square q(t)\right)\cdot h(t)$
satisfies condition \eqref{nec_condition} for all $h\in
H^{\beta}\left(I, \mathbb{R}^{d}\right)$, then function $q$
satisfies the following nondifferentiable Euler--Lagrange equation:
\begin{equation}
\label{eq:elnd}
\partial_{2} L\left(t,q(t),\square q(t)\right)-\square
\partial_{3} L\left(t,q(t),\square q(t)\right)=0\, .
\end{equation}
\end{theorem}

It is worth to mention that the Euler--Lagrange equation
\eqref{eq:elnd} can be generalized in many different ways:
see \cite{MyID:162} for the cases when the Lagrangian $L$
contains multiple scale derivatives, depends on a parameter,
or contains higher-order scale derivatives.

\begin{definition}[Nondifferentiable extremals]
\label{def:scale:ext}
The solutions $q(\cdot)$ of the nondifferentiable Euler--Lagrange equation
\eqref{eq:elnd} are called \emph{nondifferentiable extremals}.
\end{definition}

\begin{definition}
\label{def:invnd}
Functional \eqref{Pe} is said to be invariant under
the $s$-parameter group of infinitesimal transformations
\begin{equation}
\label{eq:tinf}
\begin{cases}
\bar{t} = t + s\tau(t,q) + o(s) \, ,\\
\bar{q}(t) = q(t) + s\xi(t,q) + o(s) \, ,\\
\end{cases}
\end{equation}
if
\begin{multline}
\label{eq:invnd}
0 = \frac{d}{ds}
\int_{\bar{t}(I)} L\bigg[t+s\tau(t,q(t)),q(t)+s\xi(t,q(t)),\\
\frac{\square q(t)+s\square{\xi}(t,q(t))}{1+s\square{\tau}(t,q(t))}\bigg]
\left(1+s\square{\tau}(t,q(t))\right)dt\Bigr|_{s=0}
\end{multline}
for any  subinterval $I \subseteq [a,b]$, where $\tau,\xi\in H^{\alpha}$.
\end{definition}

Lemma~\ref{thm:CNSI:SCV} establishes a necessary
condition of invariance for \eqref{Pe}. Condition \eqref{eq:cnsind}
will be used in the proof of our Noether type theorem.

\begin{lemma}[Necessary condition of invariance]
\label{thm:CNSI:SCV}
If functional \eqref{Pe} is invariant under the
one-parameter group of transformations \eqref{eq:tinf}, then
\begin{multline}
\label{eq:cnsind} \int_{t_{a}}^{t_{b}}\Bigl[\partial_{1}
L\left(t,q(t),\square q(t)\right)\tau +\partial_{2}
L\left(t,q(t),\square q(t)\right)\cdot\xi
\\
+\partial_{3}
L\left(t,q(t),\square q(t)\right)\cdot\left(\square\xi
-\square q(t)\square\tau\right)+L\left(t,q(t),\square q(t)\right)\square\tau\Bigr]dt
 = 0\,.
\end{multline}
\end{lemma}

\begin{proof}
Without loss of generality, we take $I=[t_a,t_b]$.
Equality \eqref{eq:cnsind} follows directly from
condition \eqref{eq:invnd}.
\end{proof}

\begin{definition}[Nondifferentiable constants of motion]
\label{def:leicond}
A quantity $C(t,q(t),\square q(t))$
is a \emph{nondifferentiable constant of motion}
if $C(t,q(t),\square q(t))$ is constant along all the
nondifferentiable extremals $q(\cdot) \in H^\alpha(I,\mathbb{R}^d)$,
$\frac{1}{2}\leq\alpha<1$ (\textrm{cf.} Definition~\ref{def:scale:ext}).
\end{definition}

Theorem~\ref{theo:cdrnd} generalizes the classical DuBois--Reymond optimality condition
\begin{equation*}
\partial_{1} L\left(t,q(t),\dot{q}(t)\right)
=\frac{d}{dt}\left\{L\left(t,q(t),\dot{q}(t)\right)
-\partial_{3} L\left(t,q(t),\dot{q}(t)\right)\cdot\dot{q}(t)\right\}
\end{equation*}
for Cresson's quantum problems of the calculus of variations.

\begin{theorem}[Nondifferentiable DuBois--Reymond necessary optimality condition]
\label{theo:cdrnd}
Let $\frac{1}{2}\leq\alpha<1\,.$ If $q\in H^{\alpha}\left(I,\mathbb{R}^{d}\right)$
with $\square q\in H^{\alpha}\left(I,\mathbb{R}^{d}\right)$,
then any nondifferentiable extremal $q$ satisfies
the following DuBois--Reymond necessary condition:
\begin{equation}
\label{eq:cdrnd}
\frac{\square}{\square
t}\left\{L\left(t,q,\frac{\square q}{\square
t}\right)-\partial_{3} L\left(t,q,\frac{\square q}{\square
t}\right)\cdot\frac{\square q}{\square t}\right\} =
\partial_{1} L\left(t,q,\frac{\square q}{\square
t}\right)\,.
\end{equation}
\end{theorem}

\begin{proof}
Using the linearity of the quantum derivative operator,
Theorems~\ref{theo:mult} and \ref{theo:derivada},
and the nondifferentiable Euler--Lagrange equation \eqref{eq:elnd},
we can write that
\begin{equation*}
\begin{split}
\square\Bigl\{L(t,&q,\square q)-\partial_{3}
L(t,q,\square q)\cdot\square q\Bigr\} \\
&=\partial_{1}L(t,q,\square q)
+\partial_{2}L(t,q,\square q)\cdot\square q
+\partial_{3}L(t,q,\square q)\cdot\square\square q \\
&\qquad-\square\partial_{3}L(t,q,\square q)\cdot\square q
-\partial_{3}L(t,q,\square q)\cdot\square\square q \\
&=\partial_{1}L(t,q,\square q)+\square q\cdot(\partial_{2}L(t,q,\square q)
-\square\partial_{3}L(t,q,\square q)) \\
&=\partial_{1}L(t,q,\square q) \, .
\end{split}
\end{equation*}
This concludes the proof.
\end{proof}

Our main result is the following.

\begin{theorem}[Nondifferentiable Noether's theorem]
\label{theo:tnnd}
If functional \eqref{Pe} is invariant in the sense
of Definition~\ref{def:invnd}, then
\begin{multline}
\label{eq:tnnd} C(t,q(t),\square q(t)) =
\partial_{3}L(t,q,\square q))\cdot\xi(t,q)\\
+\Bigl(L(t,q,\square q)-\partial_{3}
L(t,q,\square q)\cdot\square q\Bigr)\tau(t,q)
\end{multline}
is a nondifferentiable constant of motion (\textrm{cf.}
Definition~\ref{def:leicond}).
\end{theorem}

\begin{proof}
Noether's nondifferentiable constant of motion \eqref{eq:tnnd} follows by using
the scale DuBois--Reymond condition \eqref{eq:cdrnd}, the nondifferentiable
Euler--Lagrange equation \eqref{eq:elnd} and Theorem~\ref{theo:mult},
into the necessary condition of invariance \eqref{eq:cnsind}:
\begin{equation}
\label{eq:dtnnd}
\begin{split}
0&=\int_{t_{a}}^{t_{b}}\Bigl[\partial_{1}
L\left(t,q(t),\square q(t)\right)\tau +\partial_{2}
L\left(t,q(t),\square q(t)\right)\cdot\xi \\
& \qquad\qquad +\partial_{3}
L\left(t,q,\square q\right)\cdot\left(\square \xi
-\square q\square \tau\right)
+L\square \tau\Bigr]dt\\
&=\int_{t_{a}}^{t_{b}}\Bigl[\tau\square (L\left(t,q,\square q\right)-\partial_{3}
L\left(t,q,\square q\right)\cdot\square q)\\
&\qquad\qquad +\left(L\left(t,q,\square q\right)
-\partial_{3}
L\left(t,q,\square q\right)\cdot\square q\right)\square \tau \\
&\qquad\qquad +\xi\cdot\square \partial_{3}
L\left(t,q,\square q\right) +\partial_{3}
L\left(t,q,\square q\right)\cdot\square \xi\Bigr]dt\\
&=\int_{t_{a}}^{t_{b}}\frac{\square }{\square
t}\Big\{\partial_{3}
L\left(t,q,\square q\right)\cdot\xi+(L\left(t,q,\square q\right)-\partial_{3}
L\left(t,q,\square q\right)\cdot\square q)\tau\Big\}dt\, .
\end{split}
\end{equation}
Using formula \eqref{Barrow1} and having in mind that
\eqref{eq:dtnnd} holds for an arbitrary $[{t_{a}},{t_{b}}]
\subseteq [a,b]$, we conclude that
$L\left(t,q,\square q\right)\cdot\xi+\Bigl(L\left(t,q,\square q\right)
-\partial_{3}L\left(t,q,\square q\right)\cdot\square q\Bigr)\tau$ is constant.
\end{proof}

If the admissible functions $q$ are differentiable, then the nondifferentiable
constant of motion \eqref{eq:tnnd} reduces to classical Noether's constant of motion
\begin{equation*}
C(t,q,\dot{q}) = \partial_{3} L\left(t,q,\dot{q}\right)\cdot\xi(t,q)
+ \left( L(t,q,\dot{q}) - \partial_{3} L\left(t,q,\dot{q}\right)
\cdot \dot{q} \right) \tau(t,q).
\end{equation*}
For this reason, the term
$\partial_{3}L(t,q,\square q))$ can be seen
as the \emph{momentum} while the term
$L(t,q,\square q)-\partial_{3} L(t,q,\square q)\cdot\square q$
can be interpreted as \emph{energy}.


\section{An Application}
\label{sec:appl}

In \cite[\S 3]{Cresson:2011}, a linear Schr\"{o}dinger equation,
with particular interest in quantum mechanics, is studied.
It is proved that, under certain conditions,
solutions of the linear Schr\"{o}dinger equation
coincide with the extremals of a certain functional \eqref{Pe}
of Cresson's quantum calculus of variations. In this section we use our
nondifferentiable Noether's theorem to find constants of motion for the problem
studied in \cite[\S 3]{Cresson:2011}. Precisely, consider the following
linear Schr\"{o}dinger equation:
\begin{equation}
\label{eq:sch1}
i\bar{h}\frac{\partial \Psi(t,q)}{\partial
t}+\frac{\bar{h}^2}{2m}\sum_{j=1}^{d}\frac{\partial^2 \Psi(t,q)}{\partial q^{2}_{j}}=U(q)\Psi(t,q),
\end{equation}
where $\bar{h}=\frac{h}{2\pi}$, $h$ is the Planck constant, $m>0$ the mass of particle,
$U:\mathbb{R}\longrightarrow\mathbb{R}$, $\Psi:\mathbb{R}^d\times\mathbb{R}\longrightarrow \mathbb{C}$
is the wave function associated to the particle on $\mathcal{C}^2(\mathbb{R}^d\times\mathbb{R},\mathbb{C})$,
subject to the condition
\begin{equation*}
\frac{\square q_{k}(t)}{\square
t}=-i2\gamma\frac{\partial \ln(\Psi(t,q))}{\partial q_{k}}\,,
\quad k=1,\ldots,d,
\end{equation*}
with $\gamma=\frac{\bar{h}}{2m}\in\mathbb{R}$. In \cite[Theorem~9]{Cresson:2011} it is shown
that the solutions $q(\cdot)$ of \eqref{eq:sch1} coincide with Euler--Lagrange extremals
of functional \eqref{Pe} with Lagrangian
\begin{equation*}
L(t,q(t),\square q(t))=\frac{1}{2}m\left(\square q(t)\right)^2-U(q)\, .
\end{equation*}
The functional
\begin{equation*}
I[q(\cdot)] =\frac{1}{2}\int_a^b
\left[m\left(-i 2\gamma\sum_{k=1}^{d}\frac{\partial \ln(\Psi(t,q))}{\partial
q_{k}}\right)^2 -2U(q)\right]dt
\end{equation*}
is invariant in the sense of Definition~\ref{def:invnd} under the
symmetries $(\tau,\xi)=(c_{k},0)$, where $c_{k}$ is an arbitrary constant.
It follows from our Theorem~\ref{theo:tnnd} that
\begin{equation}
\label{eq:scm:Ex1} 2m\left(\gamma\sum_{k=1}^{d}\frac{\partial
\ln(\Psi(t,q))}{\partial q_{k}}\right)^2+U(q)=\frac{1}{8m}\left(\frac{h}{\pi}\sum_{k=1}^{d}\frac{\partial
\ln(\Psi(t,q))}{\partial q_{k}}\right)^2+U(q)
\end{equation}
is a nondifferentiable constant of motion: \eqref{eq:scm:Ex1} is preserved along
all solutions $q(t)$ of the linear Schr\"{o}dinger equation \eqref{eq:sch1}.


\section*{Acknowledgements}

This work was partially supported by Portuguese funds through
the Center for Research and Development in Mathematics and Applications (CIDMA)
and the Portuguese Foundation for Science and Technology (FCT),
within project PEst-OE/MAT/UI4106/2014, and the project ``Mathematics and Applications''
from \emph{C\^{a}mara de Investiga\c{c}\~{a}o} (CAMI), University of Cape Verde, Cape Verde.
Gast\~{a}o S. F. Frederico is also grateful to IMPA
(Instituto Nacional de Matem\'{a}tica Pura e Aplicada), Rio de Janeiro, Brasil,
for a one-month post-doc visit during February 2014. The hospitality and
the good working conditions at IMPA are here very acknowledged.
Finally, the authors would like to thank a reviewer for valuable comments.




\begin{thebibliography}{xx}

\bibitem{MyID:185}
K. A. Aldwoah, A. B. Malinowska\ and\ D. F. M. Torres,
The power quantum calculus and variational problems,
Dyn. Contin. Discrete Impuls. Syst. Ser. B Appl. Algorithms {\bf 19} (2012), no.~1-2, 93--116.
{\tt arXiv:1107.0344}

\bibitem{MyID:131}
R. Almeida\ and\ D. F. M. Torres,
H\"olderian variational problems subject to integral constraints,
J. Math. Anal. Appl. {\bf 359} (2009), no.~2, 674--681.
{\tt arXiv:0807.3076}

\bibitem{MyID:162}
R. Almeida\ and\ D. F. M. Torres,
Generalized Euler-Lagrange equations for variational problems with scale derivatives,
Lett. Math. Phys. {\bf 92} (2010), no.~3, 221--229.
{\tt arXiv:1003.3133}

\bibitem{MyID:188}
R. Almeida\ and\ D. F. M. Torres,
Nondifferentiable variational principles in terms of a quantum operator,
Math. Methods Appl. Sci. {\bf 34} (2011), no.~18, 2231--2241.
{\tt arXiv:1106.3831}

\bibitem{MR1843232}
M. Bohner\ and\ A. Peterson,
{\it Dynamic equations on time scales},
Birkh\"auser Boston, Boston, MA, 2001.

\bibitem{MyID:194}
A. M. C. Brito da Cruz, N. Martins\ and\ D. F. M. Torres,
Higher-order Hahn's quantum variational calculus,
Nonlinear Anal. {\bf 75} (2012), no.~3, 1147--1157.
{\tt arXiv:1101.3653}

\bibitem{MyID:222}
A. M. C. Brito da Cruz, N. Martins\ and\ D. F. M. Torres,
A symmetric quantum calculus.
In: {\it Differential and Difference Equations with Applications},
Springer Proceedings in Mathematics \& Statistics, Vol.~47
(Eds.: S.~Pinelas, M.~Chipot and Z.~Dosla), 2013, 359--366.
{\tt arXiv:1112.6133}

\bibitem{MyID:229}
A. M. C. Brito da Cruz, N. Martins\ and\ D. F. M. Torres,
A symmetric N\"{o}rlund sum with application to inequalities.
In: {\it Differential and Difference Equations with Applications},
Springer Proceedings in Mathematics \& Statistics, Vol.~47
(Eds.: S.~Pinelas, M.~Chipot and Z.~Dosla), 2013, 495--503.
{\tt arXiv:1203.2212}

\bibitem{MyID:246}
A. M. C. Brito da Cruz, N. Martins\ and\ D. F. M. Torres,
Hahn's symmetric quantum variational calculus,
Numer. Algebra Control Optim. {\bf 3} (2013), no.~1, 77--94.
{\tt arXiv:1209.1530}

\bibitem{MR2738030}
C. Castro,
On nonlinear quantum mechanics, noncommutative phase spaces,
fractal-scale calculus and vacuum energy,
Found. Phys. {\bf 40} (2010), no.~11, 1712--1730.

\bibitem{MR2138974}
J. Cresson,
Non-differentiable variational principles,
J. Math. Anal. Appl. {\bf 307} (2005), no.~1, 48--64.
{\tt arXiv:math/0410377}

\bibitem{MyID:111}
J. Cresson, G. S. F. Frederico\ and\ D. F. M. Torres,
Constants of motion for non-differentiable quantum variational problems,
Topol. Methods Nonlinear Anal. {\bf 33} (2009), no.~2, 217--231.
{\tt arXiv:0805.0720}

\bibitem{MR2798411}
J. Cresson\ and\ I. Greff,
A non-differentiable Noether's theorem,
J. Math. Phys. {\bf 52} (2011), no.~2, 023513, 10~pp.

\bibitem{Cresson:2011}
J. Cresson\ and\ I. Greff,
Non-differentiable embedding of Lagrangian systems and partial differential equations,
J. Math. Anal. Appl. {\bf 384} (2011), no.~2, 626--646.

\bibitem{Dubois:2012}
F. Dubois, I. Greff\ and\ T. H\'{e}lie,
On least action principles for discrete quantum scales.
In: Lecture Notes in Computer Science,
Volume 7620 LNCS, 2012, 13--23.

 \bibitem{MyID:253}
G. S. F. Frederico\ and\ D. F. M. Torres,
A nondifferentiable quantum variational embedding in presence of time delays,
Int. J. Difference Equ. {\bf 8} (2013), no.~1, 49--62.
{\tt arXiv:1211.4391}

\bibitem{MR1865777}
V. Kac\ and\ P. Cheung,
{\it Quantum calculus},
Universitext, Springer, New York, 2002.

\bibitem{MR3028184}
A. B. Malinowska\ and\ N. Martins,
Generalized transversality conditions for the Hahn quantum variational calculus,
Optimization {\bf 62} (2013), no.~3, 323--344.
{\tt arXiv:1202.0176}

\bibitem{MyID:187}
A. B. Malinowska\ and\ D. F. M. Torres,
The Hahn quantum variational calculus,
J. Optim. Theory Appl. {\bf 147} (2010), no.~3, 419--442.
{\tt arXiv:1006.3765}

\bibitem{book:quant}
A. B. Malinowska\ and\ D. F. M. Torres,
{\it Quantum variational calculus},
Springer Briefs in Electrical and Computer Engineering:
Control, Automation and Robotics, Springer, New York, 2014.

\bibitem{MyID:146}
N. Martins\ and\ D. F. M. Torres,
L'H\^{o}pital-type rules for monotonicity with application to quantum calculus,
Int. J. Math. Comput. {\bf 10} (2011), M11, 99--106.
{\tt arXiv:1011.4880}

\bibitem{MyID:220}
N. Martins\ and\ D. F. M. Torres,
Higher-order infinite horizon variational problems in discrete quantum calculus,
Comput. Math. Appl. {\bf 64} (2012), no.~7, 2166--2175.
{\tt arXiv:1112.0787}

\bibitem{Nottale:1992}
L. Nottale,
The theory of scale relativity,
Internat. J. Modern Phys. A {\bf 7} (1992), no.~20, 4899--4936.

\bibitem{Nottale:1999}
L. Nottale,
The scale-relativity program,
Chaos Solitons Fractals {\bf 10} (1999), no.~2-3, 459--468.

\bibitem{CD:JMS:Torres:2004}
D. F. M. Torres,
Proper extensions of Noether's symmetry theorem for nonsmooth
extremals of the calculus of variations,
Commun. Pure Appl. Anal. {\bf 3} (2004), no.~3, 491--500.

\end{thebibliography}
\end{document}